\documentclass[12pt]{amsart}
\usepackage{a4wide,enumerate,color}
\allowdisplaybreaks
\usepackage{mathtools}
\mathtoolsset{showonlyrefs}
\newtheorem{theo}{Theorem}

\newtheorem{lemma}{Lemma}

\newcommand{\R}{\mathbb{R}}	
\newcommand{\T}{\mathbb{T}}	
\newcommand{\eps}{\varepsilon}	
\newcommand{\pa}{\partial}		
\newcommand{\Div}{\textrm{div}\,}	

\newcommand{\na}{\nabla}		

\newcommand{\IRd}{\int_{\R^3}}

\newcommand{\cA}{\mathcal{A}}

\newcommand{\charf}[1]{{\raisebox{2pt}{\large $\chi$}}_{#1}}
\newcommand{\urho}{u^{(\rho)}}
\newcommand{\usg}{u^{(\sigma)}}
\newcommand{\prho}{p^{(\rho)}}
\newcommand{\psg}{p^{(\sigma)}}

\begin{document}

\thanks{ Acknowledgments: ED acknowledges partial support from Austrian Science Fund (FWF), grants P27352 and P30000. MPG is supported by DMS-1514761. NZ acknowledges support from the Austrian Science Fund (FWF), grants P22108, P24304, W1245.}

\title[Non-local porous media equation]{Longtime behavior and weak-strong uniqueness for a nonlocal porous media equation}

\author[E. S. Daus]{Esther S. Daus}
\address{Institute for Analysis and Scientific Computing, Vienna University of  
	Technology, Wiedner Hauptstra\ss e 8--10, 1040 Wien, Austria}
\email{esther.daus@tuwien.ac.at} 

\author[M. Gualdani]{Maria Gualdani}
\address{Department of Mathematics George Washington University, 801 22nd Street, NW Washington DC, 20052 (USA) and Department of Mathematics, Royal Institute of Technology (KTH), Lindstedtsvägen 25 10044 Stockholm (Sweden) }
\email{gualdani@gwu.edu} 

\author[N. Zamponi]{Nicola Zamponi}
\address{Institute for Analysis and Scientific Computing, Vienna University of  
	Technology, Wiedner Hauptstra\ss e 8--10, 1040 Wien, Austria}
\email{nicola.zamponi@tuwien.ac.at} 

\date{\today}


\begin{abstract}
In this manuscript we consider a non-local porous medium equation with non-local diffusion effects given by a fractional heat operator
\begin{equation*}
    \begin{cases}
      \pa_t u = \Div(u\na p),\\
      \pa_t p = -(-\Delta)^s p + u^2,
    \end{cases}
\end{equation*}
 in three space dimensions for $3/4\le s < 1$ and analyze the long time asymptotics. The proof is based on energy methods and leads to algebraic decay towards the stationary solution $u=0$ and $\nabla p=0$ in the $L^2(\mathbb{R}^3)$-norm. The decay rate depends on the exponent $s$. We also show weak-strong uniqueness of solutions and continuous dependence from the initial data.  As a side product of our analysis we also show that existence of weak solutions, previously shown in \cite{CGZ18} for $3/4\le s \le 1$, holds for $1/2 < s\le 1$ if we consider our problem in the torus.
 
\end{abstract}



\maketitle


\section{Introduction}
We consider the following porous medium equation with non-local diffusion effects: 
\begin{equation}\label{1L}
    \begin{cases}
      \pa_t u = \Div(u\na p),\\
      \pa_t p = -(-\Delta)^s p + u^2.
    \end{cases}
\end{equation}
For all $x\in \R^3$, the functions $u(x,t)\ge 0$ and $p(x,t)\ge 0$ denote respectively the density and the pressure. In a previous paper \cite{CGZ18} we have introduced the model and showed existence of weak solutions. In the current manuscript we study the long time behavior and weak-strong uniqueness. The model describes the time evolution of a density function $u$ that evolves under the continuity equation 
$$
\pa_t u = \Div(u {\bf{v}}),
$$
where the velocity is conservative, ${\bf{v}}=\nabla p$, and $p$ is related to $u^2$ by the inverse of the fractional heat operator $\pa_t + (-\Delta )^s$.
Equation \eqref{1L} is the parabolic-parabolic version of a problem recently studied in \cite{BIK15}:
\begin{align}\label{BIKeq} 
\pa_t u = \Div(|u|\nabla ^{\alpha-1}(|u|^{m-2}u)).
\end{align}
Note in fact that for $m=3$ and $\alpha = 2-2s$ equation \eqref{BIKeq} reduces to the parabolic-elliptic version of \eqref{1L}. 

As we already mentioned, existence of weak solutions to \eqref{1L} with $3/4\le s \le 1$ was recently studied in \cite{CGZ18}. The introduction of $\pa_t p$ introduced several complications due to the non-locality in time relation between $u$ and $p$. Consequence of this nonlocality is that techniques such as maximum principle and Stroock-Varopoulos inequality do not work in the current parabolic-parabolic setting. Existence results for $s < 3/4$ is still an open problem, except in the case when $x \in \mathbb{T}^3$ (see Theorem \ref{cor.ex}). 

Existence of weak solutions, regularity and finite speed of propagation for a {\em{linear}} parabolic-elliptic version of \eqref{1L}
\begin{align}
 \pa_t u & = \Div(u\na p), \quad p = (-\Delta )^{-s} u ,\quad 0<s\le 1, \label{CV} 
\end{align}
has been considered in \cite{CV11, CV15, SV14, CSV13, AS08} and long-time asymptotics in \cite{CV11_II}.  In \cite{CV11_II} the authors perform a self-similar rescaling and rewrite \eqref{CV} as a non-local Fokker-Planck equation with confinement potential. Entropy estimates lead to algebraic decay of the solution towards self-similar solutions called fractional Barenblatt functions. In this contest we also recall a very recent result \cite{AJY19} that shows that solutions to the fractional drift-diffusion-Poisson model 
\begin{align*}
 \pa_t u & = -(-\Delta )^{-\alpha} u + \Div(u\na p), \quad p = (-\Delta )^{-1} u ,\quad 0<\alpha\le 1, \label{AJY} 
\end{align*}
converge algebraically, as time grows, towards the fundamental solution to the linear fractional heat equation $\pa_t u  = -(-\Delta )^{-\alpha} u $.

System \eqref{1L} is also reminiscent to a well-studied macroscopic model proposed in \cite{GL97} for phase segregation in particle systems with long range interaction:   
\begin{equation}\label{1I}
    \begin{cases} \pa_t u  =  \Delta u + \Div(\sigma(u)\na p),\\
 p  = K \ast u.
  \end{cases}
\end{equation}
Here $\sigma(u) := u(1-u)$ denotes the mobility of the system and $K$ a bounded, symmetric and compactly supported kernel. Several variants of \eqref{1I} have been considered in the literature and we refer to \cite{R91, GL97, GLM00, GL98} and references therein for  more detailed discussions on this topic.  We also mention \cite{LMG01} for the study of a deterministic particle method for heat and Fokker–Planck equations of porous media type where the non-locality appears in the coefficients. \\

The main results of this manuscript are summarized in the following three theorems: 
\begin{theo}[Long-time behavior]\label{thr.ltb}
Let $3/4\leq s < 1$. Assume that  $u, p$ are weak solutions in the sense of Theorem 1 in \cite{CGZ18} with initial data $p_0$ that satisfies 
$\int_{\R^3}|x||\nabla p_0|^2dx < \infty$.  Let 
$$ H[u,p] := \IRd\left( u^2 + \frac{1}{2}|\na p|^2 \right)dx.  $$
There exists a constant $C>0$ such that
$$
H[u(t),p(t)] \leq C t^{-\lambda},\qquad t>1,
$$
with $$\lambda = \frac{3(1-s)}{2s(5+2s)} > 0.$$
 Consequently we have strong convergence of $(u,\nabla p)$ towards $(0,0)$ in $L^2(\R^3)$ with algebraic decay rate $t^{-\lambda/2}$.
\end{theo}

The main idea of the proof of Theorem \ref{thr.ltb} relies on entropy methods. Throughout this entire paper we will denote with $C$ any generic positive constant independent of $T$.  
The functional 
$ H[u,p] = \IRd\left( u^2 + \frac{1}{2}|\na p|^2 \right)dx  $
is a Lyapunov functional for \eqref{1L} and satisfies the bound
$$
 \IRd\left( u^2 + \frac{1}{2}|\na p|^2 \right)dx + \int_0^T \IRd |(-\Delta)^{s/2}\na p |^2 dx dt=  \IRd\left( u_{in}^2 + \frac{1}{2}|\na p_{in}|^2 \right)dx. 
$$
Indeed, formal computations show that 
\begin{align*}
\frac{d}{dt}\IRd u^2 dx &= \langle \Div(u\na p) , 2 u\rangle =  -\IRd \na u^{2}\cdot\na p dx\\
&=-\frac{d}{dt}\IRd\frac{|\na p|^2}{2}dx - \IRd |(-\Delta)^{s/2}\na p |^2 dx ,
\end{align*}
after testing the equation for $p$ against $\Delta p$.
This leads to 
\begin{equation}\label{entropy.diss}
\frac{d}{dt}H[u,p] + \IRd |(-\Delta)^{s/2}\na p |^2 dx = 0,\qquad t>0.
\end{equation}
The key (and new!)  observation that leads to the proof of decay is that the expression $\int_0^T \int_{\mathbb{R}^3} \na (u^2)\cdot \na p \;dxdt$ defines a scalar product $\cA(\cdot,\cdot)$, namely
\begin{align*}
      \int_0^T \int_{\mathbb{R}^3} \na (u^2)\cdot \na p dxdt = \cA(\na (u^2), \na (u^2)),
\end{align*}
and any sequence that is Cauchy in the $\cA$-norm converges almost everywhere. Moreover by writing $\cA(\na (u^2), \na (u^2))$ in terms of Fourier transform we get an improved bound  $\|u\|^2_{L^2(0,T,L^2(\mathbb{R}^3)}\le T^{\alpha}$ with $\alpha<1$. This  combined with a sharper estimate for $\|\nabla p\|^2_{L^2(0,T,L^2(\mathbb{R}^3)}$ yields our algebraic decay. \\

Our second main theorem concerns a weak-strong uniqueness result:
\begin{theo}[Weak-strong uniqueness, continuous dependence on data]
\label{thr.ws}
Let $\frac{3}{4}\leq s\leq 1$. Assume that $v$ is a strong solution to 
\begin{align*}
 \pa_t v = \Div(v\nabla q),\quad 
 \pa_t q + (-\Delta)^s q = v^2\quad &\mbox{in }\R^3\times (0,\infty),\\
 v(0) = v_0\quad &\mbox{in }\R^3 ,
\end{align*}
such that 
$$\exists\nu>0 : ~~
\nabla v \in L^\infty(0,\infty; L^{\frac{12}{3+2s} + \nu}(\R^3)),\quad 
\sup_{\R^3\times (0,\infty)}\Delta q < \infty .
$$
Then there exists a constant $K>0$ such that, for any $u$ weak solution to
\eqref{1L} according to Theorem 1 in \cite{CGZ18}:
$$
H[(u(t),p(t))\vert (v(t),q(t))]\leq e^{Kt}H[(u_0,p_0)\vert (v_0,q_0)],\qquad t>0,
$$
where $H[(u,p)\vert (v,q)]$ denotes the relative entropy between $u$ and $v$:
\begin{align*}
 H[(u,p)\vert (v,q)] = \IRd\left( (u-v)^2 + \frac{1}{2}|\nabla (p-q)|^2 \right)dx .
\end{align*}
In particular $u\equiv v$ if $u_0\equiv v_0$. This means that if there exists a strong solution, then any weak solution with the same initial data coincides with it.
\end{theo}
The weak-strong uniqueness is a familiar concept in the field of fluid-dynamic equations and conservation laws \cite{L34, P59, S62, D10}. It  is not a uniqueness result in the standard form: it states in fact that if strong solutions exist (still an open question for \eqref{1L}), then they are unique even when compared to all weak solutions. As in the case of fluid-dynamic equations, our notion of weak solution includes an energy inequality and such energy inequality is fundamental for the proof of Theorem \ref{thr.ws} (for the case of Navier-Stokes equation Scheffer and Shnirelman gave a  counterexample to weak-strong uniqueness if bounds for the energy functional are removed). \\

Before stating our last result we recall for completeness the existence theorem for weak solutions proven in \cite{CGZ18}:

\begin{theo} \label{Prior_thm} (Theorem 1 in \cite{CGZ18}) Let $3/4 \leq s\leq 1$ and $u_{0}, p_{0} : \R^3 \to (0,+\infty)$ be functions  such that $u_{0}, p_{0}\in L^1(\R^3)$,  
$ \int_{\R^3} u^2_{0} + |\nabla  p_{0}|^2 \;dx <+\infty$ and $ \int_{\R^3} u_{0}\gamma(x) \;dx<  +\infty$, with $\gamma(x) := \sqrt{1+|x|^2}$. 
Let $q>3/s.$ There exist functions $u, p: \R^3 \times [0,{\infty)} \to [0,+\infty) $ such that for every $T>0$ 
\begin{align*}
u\in L^\infty(0,T,L^1\cap L^2(\R^3)), &\quad p\in L^\infty(0,T,H^{1}\cap L^1(\R^3)), \\
p\in  L^2(0,T,H^{s+1}(\R^3)), & \quad {\sup_{[0,T]}}\int_{\R^3} u\gamma \;dx < +\infty, \\
\pa_t u  \in L^2(0,T,(W^{1,q}(\R^3))'), & \quad  \pa_t p\in L^2(0,T,(L^2\cap L^4(\R^3))'),
\end{align*}
 which satisfy the following weak formulation of (\ref{1L}): 
\begin{align*}
& \int_0^T  \langle \pa_t u, \phi\rangle dt 
+\int_0^T\int_{\R^{3}}u\nabla p\cdot\nabla\phi\, dx dt = 0\quad \forall \phi \in {L^2(0,T; W^{1,q}(\R^3))}, \\ 
&\int_0^T  \langle \pa_t p, \psi\rangle  dt  + \int_0^T\int_{\R^3}( (-\Delta)^{s} p - {u^2}) \psi dx dt= 0
\quad\forall \psi \in {L^2(0,T; L^2\cap L^4(\R^3))},\\
&\lim_{t\to 0} u(t) = u_{0}\quad\textrm{in} \; W^{1,q}(\R^3)',\quad  \lim_{t\to 0} p(t) = p_{0} \quad\textrm{in}\; (L^2\cap L^4(\R^3))'  .
\end{align*}
\end{theo}

Here is our extension of Theorem \ref{Prior_thm} to the torus case: 

\begin{theo}[Existence of solutions, torus case]\label{cor.ex}
Same assumptions as in Theorem \ref{Prior_thm}, with the exception that $1/2<s\leq 1$. Then there exists $u : \T^3\times [0,\infty)\to [0,\infty)$
weak solution to \eqref{1L} with $\R^3$ replaced by $\T^3$.
\end{theo}

The rest of the manuscript is divided into three sections: Section \ref{sec:thm1} contains proof of Theorem \ref{thr.ltb}, Section \ref{sec:thm2} the one of Theorem \ref{thr.ws} and Section \ref{sec:thm4} the proof of Theorem \ref{cor.ex}.

\section{Proof of Theorem \ref{thr.ltb}}\label{sec:thm1}

We first show the following auxiliary result.
%
%
%
\begin{lemma}\label{lem.xH}
Under the same assumptions of Theorem ~\ref{thr.ltb} there exists a constant $C>0$ such that 
$$
\sup_{t\in [0,T]}\IRd |x|\left( u^2(x,t) + \frac{1}{2}|\nabla p(x,t)|^2 \right)dx \leq C T^{1/2} ,\qquad T>1.
$$
\end{lemma}
\begin{proof}[Proof of Lemma \eqref{lem.xH}] The proof is divided into several steps.\medskip\\
{\em Step 1: bound for $\|p(t)\|_{L^2(\R^3)}$.}
By using $p$ as a test function in the second equation of \eqref{1L} we get
\begin{align}\label{est.p2.1}
 &\frac{1}{2}\|p(T)\|_{L^2(\R^3)}^2 - \frac{1}{2}\|p_0\|_{L^2(\R^3)}^2 + 
 \int_0^T\IRd |(-\Delta)^{s/2}p|^2 dx dt = \int_0^T\IRd u^2 p dx dt\\
 \nonumber
 &\leq \int_0^T\|p(t)\|_{L^\infty(\R^3)}\IRd u^2 dx dt\leq C \int_0^T\|p(t)\|_{L^\infty(\R^3)}dt,
\end{align}
since $u\in L^\infty(0,\infty; L^2(\R^3))$ thanks to the entropy inequality \eqref{entropy.diss}.
Moreover, given that $6/(3-2s)>3$, the following Gagliardo-Nirenberg inequality holds
\begin{align}\label{GN.pinf}
 \|p\|_{L^\infty(\R^3)} \leq C \|p\|_{L^1(\R^3)}^{1-\xi}\|\nabla p\|_{L^{6/(3-2s)}(\R^3)}^\xi ,
\end{align}
for some exponent $\xi\in (0,1)$. On the other hand, integrating the second equation in \eqref{1L} in $\R^3\times [0,T]$ yields
$$ \|p(T)\|_{L^1(\R^3)} - \|p_0\|_{L^1(\R^3)} = \int_0^T\IRd u^2 dx dt , $$
which implies (again thanks to the entropy inequality \eqref{entropy.diss})
\begin{align}
 \|p(T)\|_{L^1(\R^3)}\leq C T,\qquad T>1.\label{est.p1}
\end{align}
From \eqref{GN.pinf} and \eqref{est.p1} it follows
\begin{align*}
 \int_0^T\|p(t)\|_{L^\infty(\R^3)}dt &\leq C T^{1-\xi}\int_0^T\|\nabla p\|_{L^{6/(3-2s)}(\R^3)}^\xi dt.
\end{align*}
By applying H\"older's inequality we get
\begin{align*}
 \int_0^T\|p(t)\|_{L^\infty(\R^3)}dt &\leq C T^{1-\xi}\left(\int_0^T\|\nabla p\|_{L^{6/(3-2s)}(\R^3)}^2 dt\right)^{\xi/2}T^{1-\xi/2}.
\end{align*}
However, Sobolev's embedding $H^s(\R^3)\hookrightarrow L^{6/(3-2s)}(\R^3)$ and the entropy inequality \eqref{entropy.diss} 
allow us to write
\begin{align}\label{est.nap.Sob}
 \int_0^T\|\nabla p\|_{L^{6/(3-2s)}(\R^3)}^2 dt\leq C\int_0^T\|\nabla (-\Delta)^{s/2} p\|_{L^{2}(\R^3)}^2 dt\leq C ,
\end{align}
which implies
\begin{align*}
 \int_0^T\|p(t)\|_{L^\infty(\R^3)}dt &\leq C T^{2-3\xi/2}\leq C T^2, \qquad T>1.
\end{align*}
The above inequality and \eqref{est.p2.1} lead to
\begin{align}
 \|p(T)\|_{L^2(\R^3)}\leq C T,\qquad T>1.\label{est.p2}
\end{align}
{\em Second step: bound for $\int_0^T\IRd u^3 dx dt$.}
It was already shown in \cite{CGZ18} that $\int_0^T\IRd u^3 dx dt  \le C(T)$. Now we need a more accurate estimate on the generic constant  $C(T)$. For that consider
\begin{align}\label{est.u3.1}
 \int_0^T\IRd u^3 dx dt &= \IRd u(T)p(T) dx - \IRd u_0 p_0 dx + \int_0^T\IRd u|\nabla p|^2 dx dt \\
 &+ \int_0^T\IRd u (-\Delta)^s p\, dx dt. \nonumber
\end{align}
From \eqref{entropy.diss}, \eqref{est.p2} we get
\begin{align}\label{est.up}
 \IRd u(T)p(T) dx\leq \|u\|_{L^\infty(0,\infty; L^2(\R^3))}\|p(T)\|_{L^2(\R^3)}\leq C T, \qquad T>1.
\end{align}
H\"older inequality yields
\begin{align}\label{est.unap2.1}
 \int_0^T\IRd u|\nabla p|^2 dx dt \leq \int_0^T\left( \IRd u^{\frac{3}{2s}} dx\right)^{\frac{2s}{3}}
 \left( \IRd |\nabla p|^{\frac{6}{3-2s}}dx \right)^{1-\frac{2s}{3}}dt .
\end{align}
The mass conservation and \eqref{entropy.diss} yields $u\in L^\infty(0,\infty; L^1\cap L^2(\R^3))$;
given that $3/4\leq s < 1$, an interpolation argument implies $u\in L^\infty(0,\infty; L^{3/2s}(\R^3))$.
This relation, together with \eqref{est.nap.Sob} and \eqref{est.unap2.1}, yields
\begin{align}\label{est.unap2}
 \int_0^T\IRd u|\nabla p|^2 dx dt \leq C.
\end{align}
Let us apply H\"older inequality to
\begin{align}\label{est.uDeltap.1}
 \int_0^T\IRd u (-\Delta)^s p dx dt \leq C\left(\int_0^T\|u\|_{L^{\frac{6}{5-2s}}(\R^3)}^2 dt\right)^{\frac{1}{2}}
 \left(\int_0^T\|(-\Delta)^s p\|_{L^{\frac{6}{1+2s}}(\R^3)}^2 dt\right)^{\frac{1}{2}}.
\end{align}
Using Sobolev's embedding $H^{1+s}(\R^3)\hookrightarrow W^{2s,6/(1+2s)}(\R^3)$ and \eqref{est.nap.Sob} we can write
\begin{align}\label{uff.1}
 \int_0^T\|(-\Delta)^s p\|_{L^{6/(1+2s)}(\R^3)}^2 dt\leq C \int_0^T\|\nabla (-\Delta)^{s/2} p\|_{L^{2}(\R^3)}^2 dt \leq C .
\end{align}
On the other hand, since $1\leq \frac{6}{5-2s}\leq 2$ and $u\in L^\infty(0,\infty; L^1\cap L^2(\R^3))$, we deduce by interpolation
$$ \int_0^T\|u\|_{L^{\frac{6}{5-2s}}(\R^3)}^2 dt\leq C T , $$
which, together with \eqref{est.uDeltap.1} and \eqref{uff.1}, yields
\begin{align}
 \int_0^T\IRd u (-\Delta)^s p dx dt \leq C T^{1/2},\qquad T>1.\label{est.uDeltap}
\end{align}
From \eqref{est.u3.1}, \eqref{est.up}, \eqref{est.unap2}, \eqref{est.uDeltap} we conclude
\begin{align}
 \int_0^T\IRd u^3 dx dt \leq C T,\qquad T>1.\label{est.u3}
\end{align}
{\em Step 3: bound for $\IRd |x|\left( u^2 + \frac{1}{2}|\nabla p|^2 \right)dx$.}
Let us compute
\begin{align*}
 \frac{d}{dt}\IRd |x| u^2 dx &= -2\IRd\frac{x}{|x|}\cdot u^2\nabla p dx - \IRd |x| \nabla u^2\cdot\nabla p dx ,\\
 \frac{d}{dt}\IRd \frac{|x|}{2} |\nabla p|^2 dx &= -\IRd |x|\nabla p\cdot\nabla (-\Delta)^s p dx + \IRd |x| \nabla u^2\cdot\nabla p dx .
\end{align*}
Adding the above identities and integrating the resulting equation in the time interval $[0,T]$ we get 
\begin{align}\label{est.xh.1}
 \IRd |x|\left( u^2(T) + \frac{1}{2}|\nabla p(T)|^2 \right)dx 
 = \IRd |x|\left( u_0^2 + \frac{1}{2}|\nabla p_0|^2 \right)dx + I_1 + I_2,\\
 \nonumber
 I_1 = -2\int_0^T\IRd\frac{x}{|x|}\cdot u^2\nabla p dx dt,\quad I_2 = -\int_0^T\IRd |x|\nabla p\cdot\nabla (-\Delta)^s p dx dt .
\end{align}
Let us now estimate $I_1$. H\"older inequality yields
\begin{align*}
 \frac{1}{2}I_1 &\leq \int_0^T\IRd u^2 |\nabla p| dx dt\leq 
 \int_0^T\|\nabla p\|_{L^{\frac{6}{3-2s}}(\R^3)}\|u\|^2_{L^{\frac{12}{3+2s}}(\R^3)}dt\\
 &\leq \left(\int_0^T\|\nabla p\|_{L^{\frac{6}{3-2s}}(\R^3)}^2 dt\right)^{1/2}
 \left(\int_0^T\|u\|_{L^{\frac{12}{3+2s}}(\R^3)}^4 dt\right)^{1/2} .
\end{align*}
Thanks to \eqref{est.nap.Sob} we deduce
\begin{align}\label{I1.1}
 I_1 \leq C \left(\int_0^T\|u\|_{L^{\frac{12}{3+2s}}(\R^3)}^4 dt\right)^{1/2} \leq C \left(\int_0^T\|u\|_{L^{3}(\R^3)}^{6-4s} dt\right)^{1/2}
\end{align}
using the interpolation 
\begin{align*}
 \|u\|_{L^{\frac{12}{3+2s}}(\R^3)} \leq \|u\|_{L^2(\R^3)}^{1-\omega}\|u\|_{L^3(\R^3)}^\omega,\quad \omega = \frac{3-2s}{2}.
\end{align*}
Note that  $6-4s\leq 3$ if $s\geq 3/4$. Therefore from H\"older's inequality and \eqref{est.u3} it follows
\begin{align}\label{I1}
 I_1\leq C \left(\int_0^T\|u\|_{L^{3}(\R^3)}^{3} dt\right)^{\frac{3-2s}{3}} T^{\frac{4s-3}{6}}\leq 
 C T^{\frac{3-2s}{3}} T^{\frac{4s-3}{6}} = C T^{1/2} .
\end{align}
Let us now consider $I_2$:
\begin{align}\label{I2.1}
 I_2 &= -\int_0^T\IRd |x|\nabla p\cdot\nabla (-\Delta)^s p dx dt \\ \nonumber
 &= -\int_0^T\IRd (-\Delta)^{s/2}(|x|\nabla p)\cdot (-\Delta)^{s/2}\nabla p\, dx dt .
\end{align}
Let us compute
\begin{align*}
 (-\Delta)^{s/2}(|x|\nabla p) &= \IRd \frac{|x|\nabla p(x)-|y|\nabla p(y)}{|x-y|^{3+s}}dy\\
 &= |x|\IRd \frac{\nabla p(x)-\nabla p(y)}{|x-y|^{3+s}}dy
 + \IRd \frac{|x|-|y|}{|x-y|^{3+s}}\nabla p(y) dy\\
 &= |x|(-\Delta)^{s/2}\nabla p + \IRd \frac{|x|-|y|}{|x-y|^{3+s}}\nabla p(y) dy .
\end{align*}
Therefore \eqref{I2.1} becomes
\begin{align*}
 I_2 &= -\int_0^T\IRd |x| |(-\Delta)^{s/2}\nabla p|^2 dx dt \\
& - \int_0^T\IRd\IRd \frac{|x|-|y|}{|x-y|^{3+s}}\nabla p(y) dy\cdot (-\Delta)^{s/2}\nabla p(x) dx dt .
\end{align*}
By Young's inequality,
\begin{align*}
 I_2 &+ \int_0^T\IRd |x| |(-\Delta)^{s/2}\nabla p|^2 dx dt\\
 &\leq \int_0^T\IRd\IRd \frac{|\nabla p(y)|}{|x-y|^{2+s}}dy\cdot |(-\Delta)^{s/2}\nabla p(x)| dx dt\\
 &\leq \frac{1}{2}\int_0^T\IRd (1+|x|) |(-\Delta)^{s/2}\nabla p|^2 dx dt
 + \frac{1}{2}\int_0^T\IRd\IRd \left| \IRd \frac{|\nabla p(y)|}{|x-y|^{2+s}}dy \right|^2 \frac{dxdt}{1+|x|}.
\end{align*}
Thanks to \eqref{est.nap.Sob} we deduce
\begin{align}\nonumber
 & I_2 \leq C + C\int_0^T\IRd |f_1\ast\nabla p|^2 \frac{dx dt}{1+|x|} + C\int_0^T\IRd |f_2\ast\nabla p|^2 \frac{dx dt}{1+|x|},\\
 & f_1(x) = |x|^{-2-s}\charf{B}(x),\quad f_2(x) = |x|^{-2-s}\charf{\R^3\setminus B}(x) ,\label{f12}
\end{align}
where $B$ is the ball of center $0$ and radius $1$.
For $\eps>0$ small enough and $i=1,2$ H\"older's inequality yields
\begin{align*}
 \IRd |f_i\ast\nabla p|^2 \frac{dx}{1+|x|}
 &\leq \|f_i\ast\nabla p\|_{L^{3-\eps}(\R^3)}^2\|(1+|\cdot |)^{-1}\|_{L^{\frac{3-\eps}{1-\eps}}(\R^3)} \\
 &\leq C\|f_i\ast\nabla p\|_{L^{3-\eps}(\R^3)}^2 .
\end{align*}
Therefore
\begin{align}\label{I2.2}
 I_2 \leq C + C\int_0^T\sum_{i=1}^2\|f_i\ast\nabla p\|_{L^{3-\eps}(\R^3)}^2 dt .
\end{align}
Let us bound the term
\begin{align*}
 \|f_1\ast\nabla p\|_{L^{3-\eps}(\R^3)}\leq C \|f_1\|_{L^q(\R^3)}\|\nabla p\|_{L^{6/(3-2s)}(\R^3)},
\end{align*}
where $q=q(\eps)\geq 1$ satisfies
$$\frac{1}{q(\eps)} + \frac{3-2s}{6} = 1 + \frac{1}{3-\eps}. $$
Note that  $q(0) = \frac{6}{5+2s}$, so $(2+s)q(0)<3$. By continuity it follows that $(2+s)q(\eps)<3$ for $\eps>0$ small enough. 
As a consequence $f_1\in L^q(\R^3)$ and
\begin{align}
\exists\eps>0 :\quad 
 \int_0^T\|f_1\ast\nabla p\|_{L^{3-\eps}(\R^3)}^2 dt \leq C\int_0^T\|\nabla p\|_{L^{6/(3-2s)}(\R^3)}^2 dt \leq C\label{est.f1}
\end{align}
thanks to \eqref{est.nap.Sob}. Let us now consider, for $\eps>0$ small enough,
\begin{align*}
 \|f_2\ast\nabla p\|_{L^{3-\eps}(\R^3)}\leq C \|f_2\|_{L^{\frac{3+\eps}{2+s}}(\R^3)}\|\nabla p\|_{L^{\kappa}(\R^3)},
\end{align*}
where $\kappa=\kappa(\eps)$ satisfies
$$ \frac{2+s}{3+\eps} + \frac{1}{\kappa(\eps)} = 1 + \frac{1}{3-\eps} . $$
Given that $\eps>0$, we have $f_2\in L^{\frac{3+\eps}{2+s}}(\R^3)$, so
\begin{align*}
 \|f_2\ast\nabla p\|_{L^{3-\eps}(\R^3)}\leq C \|\nabla p\|_{L^{\kappa}(\R^3)}.
\end{align*}
We observe that $\kappa(0) = \frac{3}{2-s}$. Gagliardo-Nirenberg inequality yields
\begin{align*}
 \|\nabla p\|_{L^{\kappa(0)}(\R^3)}\leq C \|p\|_{L^1}^{1-\eta}\|\nabla p\|_{L^{6/(3-2s)}(\R^3)}^\eta,
 \quad \eta = \frac{4+2s}{5+2s}.
\end{align*}
From the above inequality and \eqref{est.p1} it follows
\begin{align*}
 \int_0^T\|\nabla p\|_{L^{\kappa(0)}(\R^3)}^2 dt \leq C T^{2(1-\eta)}\int_0^T\|\nabla p\|_{L^{6/(3-2s)}(\R^3)}^{2\eta} dt ,
\end{align*}
and by applying H\"older's inequality and \eqref{est.nap.Sob} we get
\begin{align*}
 \int_0^T\|\nabla p\|_{L^{\kappa(0)}(\R^3)}^2 dt \leq C T^{3(1-\eta)}\left(\int_0^T\|\nabla p\|_{L^{6/(3-2s)}(\R^3)}^{2} dt\right)^{\eta}
 \leq C T^{3(1-\eta)} .
\end{align*}
It holds $3(1-\eta) = \frac{3}{5+2s}\leq\frac{6}{13}<\frac{1}{2}$ since $s\geq\frac{3}{4}$. 
If $\eps>0$ is small enough, by arguing in the same way one can show 
\begin{align*}
 \int_0^T\|\nabla p\|_{L^{\kappa(\eps)}(\R^3)}^2 dt \leq C T^{3(1-\eta(\eps))}
\end{align*}
for some $\eta(\eps)\in [0,1]$ such that $3(1-\eta(\eps))<1/2$ (since $\eta(\eps)$ is continuous). 
We conclude 
\begin{align}
\exists\eps>0 :\quad 
 \int_0^T\|f_2\ast\nabla p\|_{L^{3-\eps}(\R^3)}^2 dt \leq C\int_0^T\|\nabla p\|_{L^{\kappa(\eps)}(\R^3)}^2 dt \leq C T^{1/2} . \label{est.f2}
\end{align}
From \eqref{I2.2}--\eqref{est.f2} we obtain
\begin{align}
 I_2 \leq C T^{1/2},\quad T>1.\label{I2}
\end{align}
The Lemma's statement follows from \eqref{est.xh.1}, \eqref{I1}, \eqref{I2}. This finishes the proof.
\end{proof}
We are now ready to prove Theorem ~\ref{thr.ltb}.
\begin{proof}[Proof of Theorem ~\ref{thr.ltb}]
 By using $2u$ as a test function in the density equation we get that
\begin{align*}
      \int_{\mathbb{R}^3} u^2 dx|_{t=0}^{t=T} + \int_0^T \int_{\mathbb{R}^3} \na (u^2)\cdot \na p dxdt =0.
\end{align*}
In \cite{CGZ18} it was shown that
\begin{align*}
      \int_0^T \int_{\mathbb{R}^3} \na (u^2)\cdot \na p dxdt = \cA[\na (u^2), \na (u^2)],
\end{align*}
where $\cA(\cdot,\cdot)$ defines a scalar product, and any sequence that is Cauchy in the $\cA$-norm converges almost everywhere. 
By applying the representation of $\cA$ in terms of the Fourier transform, we get that
\begin{align}\label{est.1}
      T\sum_{m=0}^{\infty}\int_{\R^3} \frac{|k|^{2s}(1-e^{-|k|^{2s}T})}{|k|^{4s} + m^2/T}|w_m(k)|^2 dk \leq C,
\end{align}
where $w_m(k)=\left[\mathcal{F}_t \mathcal{F}_x \na (u^2)\right]_m(k)$, where $\mathcal{F}_x$ denotes the Fourier-transform in space and $\mathcal{F}_t$ the Fourier-transform in time. For $m=0$ in \eqref{est.1} we get 
\begin{align}\label{est.2}
      T\int_{\R^3}\frac{1-e^{-|k|^{2s}T}}{|k|^{2s}}|w_0(k)|^2 dk \leq C,
\end{align}
where 
\begin{align*}
 w_0(k)=\frac{1}{\sqrt{2T}}\int_{-T}^{T}ik\hat{u^2}(k,t)dk = \sqrt{\frac{2}{T}}ik\hat{U}(k,T),
\end{align*}
where $U(k,t):= \int_0^t u^2(x,\tau)d\tau$ and $\hat{\cdot}=\mathcal{F}_x$ denotes the Fourier transform with respect to $x$, where we extended $u$ as even function of $t$. Thus, \eqref{est.2} implies 
\begin{align}\label{est.3}
 \int_{\R^3} \frac{1-e^{-|k|^{2s}T}}{|k|^{2s-2}}|\hat U(k,T)|^2dk \leq C.
\end{align}
Now we consider
\begin{align*}
 \int_{\R^3}\frac{e^{-|k|^{2s}T}}{|k|^{2s-2}}|\hat{U}(k,T)|^2 dk \leq \|\hat{U}\|^2_{L^{\infty}(\mathbb{R}^3)} \int_{\mathbb{R}^3}e^{-|k|^{2s}T}|k|^{2-2s}dk.
\end{align*}
However, it holds 
\begin{align*}
 \|\hat{U}(T)\|^2_{L^{\infty}(\R^3)}\leq C\| U(T) \|_{L^1(\R^3)}^2 = C \left(\int_0^T \int_{\R^3}u^2 dxdt\right)^2 \leq CT^2
\end{align*}
thanks to the entropy inequality \eqref{entropy.diss}. Consequently, it follows by using the rescaling $\tilde{k}=kT^{1/(2s)}$ 
\begin{align*}
\int_{\R^3}\frac{e^{-|k|^{2s}T}}{|k|^{2s-2}}|\hat{U}(k,T)|^2 dk \leq CT^2 \int_{\R^3}\frac{e^{-|k|^{2s}T}}{|k|^{2s-2}}dk 
= C T^{3-\frac{5}{2s}}\int_{\R^3}\frac{e^{{-|\tilde{k}|}^{2s}}}{|\tilde{k}|^{2s-2}}d\tilde{k} \leq C T^{3-\frac{5}{2s}} .
\end{align*}
We deduce from \eqref{est.3} that
\begin{align}\label{est.4}
    \int_{\R^3}|k|^{2(1-s)}|\hat{U}(k,T)|^2 dk \leq C T^{3-\frac{5}{2s}}.
\end{align}
Net now $R=R(T)>0$ be a generic constant depending on $T$. Then \eqref{est.4} implies 
$$R^{2(1-s)}\int_{B_R^c}|\hat{U}(k,T)|^2 dk \leq C T^{3-\frac{5}{2s}}.$$
On the other hand, it holds for a ball $B_R$ of radius $R>0$ centered around the origin
\begin{align*}
    \int_{B_R}|\hat{U}(k,T)|^2 dk \leq |B_{R}|\|\hat{U}(T)\|^2_{L^{\infty}(\R^3)}\leq C R^3 T^2.
\end{align*}
Thus, it follows 
\begin{align*}
 \int_{\R^3}|\hat{U}(k,T)|^2dx \leq C(R^{-2(1-s)}T^{3-\frac{5}{2s}} + R^3T^2).
\end{align*}
We now minimize the right-hand side: we choose $R=R(T)>0$ such that 
\begin{align*}
 0 = \frac{d}{d\rho}(\rho^{-2(1-s)}T^{3-\frac{5}{2s}} + \rho^3 T^2)_{\vert_{\rho=R}}
 = -2(1-s)R^{-3+2s}T^{3-\frac{5}{2s}} + 3R^{2}T^2
\end{align*}
that is
$$ R = c T^{-1/2s} . $$
It follows
\begin{align*}
 \int_{\R^3}|\hat{U}(k,T)|^2dx \leq C T^{2 - 3/2s} .
\end{align*}
The fact that the Fourier transform is an isometry $L^2\to L^2$ and the definition of $U$ imply
\begin{align}\label{thg.1}
 \int_{\R^3}\left( \int_0^T u^2(x,t)dt\right)^2 dx \leq C T^{2 - 3/2s} .
\end{align}
From \eqref{thg.1} and Jensen's inequality it follows
\begin{align*}
 \frac{C T^{2 - 3/2s}}{|B_R|} \geq \frac{1}{|B_R|}\int_{B_R}\left( \int_0^T u^2(x,t)dt\right)^2 dx
 \geq \left(\frac{1}{|B_R|}\int_{B_R}\int_0^T u^2(x,t)dt dx\right)^2
\end{align*}
and so
\begin{align}\label{thg.2}
 \int_0^T \int_{B_R}u^2(x,t)dx dt \leq C R^{3/2} T^{1-3/4s}\quad \forall R>0.
\end{align}
However, Lemma 1 implies that 
\begin{align}\label{thg.3}
 \int_0^T \int_{\R^3\setminus B_R}u^2(x,t)dx dt \leq 
 \frac{T}{R}\sup_{t\in [0,T]}\int_{\R^3\setminus B_R} |x|\left( u^2(x,t) + \frac{1}{2}|\nabla p(x,t)|^2 \right)dx
 \leq \frac{T^{3/2}}{R} .
\end{align}
Summing \eqref{thg.2} and \eqref{thg.3} leads to
\begin{align}\label{thg.4}
 \int_0^T\IRd u^2 dx dt \leq C\left( R^{3/2}T^{1-3/4s} + R^{-1}T^{3/2} \right)\quad T\geq 1,~~ R>0.
\end{align}
Again, we choose $R=R(T)$ such that the right-hand side of \eqref{thg.4} is minimal, which yields $R=c T^{1/5+3/10s}$.
It follows
\begin{align}\label{est.u}
 \int_0^T\IRd u^2 dx dt \leq C T^{\frac{13-3/s}{10}},\quad T\geq 1.
\end{align}
Let us now find a similar estimate for $\nabla p$. Gagliardo-Nirenberg inequality leads to
\begin{equation}
 \|\nabla p\|_{L^2(\R^3)}\leq C \|p\|_{L^1(\R^3)}^{1-\theta}\|\nabla p\|_{L^{6/(3-25)}(\R^3)}^\theta , \quad 
\theta = \frac{5}{5+2s}.\label{GN.nap}
\end{equation}
Taking the power $2/\theta$ of both members of the above inequality and integrating it in time yield
\begin{align*}
\int_0^T\|\nabla p\|_{L^2(\R^3)}^{2/\theta}dt &\leq C \int_0^T\|p\|_{L^1(\R^3)}^{2(1-\theta)/\theta}\|\nabla p\|_{L^{6/(3-25)}(\R^3)}^2 dt\\
&\leq C \|p\|_{L^\infty(0,T; L^1(\R^3))}^{2(1-\theta)/\theta}\int_0^T\|\nabla p\|_{L^{6/(3-25)}(\R^3)}^2 dt .
\end{align*}
From \eqref{est.p1} it follows
\begin{align*}
 \int_0^T\|\nabla p\|_{L^2(\R^3)}^{2/\theta}dt
 \leq C T^{2(1-\theta)/\theta}\int_0^T\|\nabla p\|_{L^{6/(3-25)}(\R^3)}^2 dt .
\end{align*}
Sobolev's embedding $H^{s}(\R^3)\hookrightarrow L^{6/(3-25)}(\R^3)$ leads to
\begin{align*}
 \int_0^T\|\nabla p\|_{L^2(\R^3)}^{2/\theta}dt
 \leq C T^{2(1-\theta)/\theta}\int_0^T\|\nabla (-\Delta)^{s/2} p\|_{L^{2}(\R^3)}^2 dt ,
\end{align*}
while \eqref{entropy.diss} allows us to deduce
\begin{align*}
 \int_0^T\|\nabla p\|_{L^2(\R^3)}^{2/\theta}dt
 \leq C T^{2(1-\theta)/\theta} \qquad T>1.
\end{align*}
The definition \eqref{GN.nap} of $\theta$ implies
\begin{align}\label{est.nap}
 \int_0^T\|\nabla p\|_{L^2(\R^3)}^{(10+4s)/5}dt
 \leq C T^{4s/5} \qquad T>1.
\end{align}
By summing \eqref{est.u} and \eqref{est.nap} and noticing that $\frac{13-3/s}{10}> \frac{4s}{5}$ for $\frac{3}{4}\leq s \leq 1$ we obtain
\begin{align*}
 \int_0^T\left(\|u\|_{L^2(\R^3)}^2 + \|\nabla p\|_{L^2(\R^3)}^{(10+4s)/5} \right)dt 
 \leq C (T^{\frac{13-3/s}{10}} +  T^{\frac{4s}{5}})\leq C T^{\frac{13-3/s}{10}} \qquad T>1.
\end{align*}
Since $(10+4s)/5 \geq 2$ and $u \in L^\infty(0,\infty; L^2(\R^3))$ it follows that
\begin{align*}
 \int_0^T\left(\|u\|_{L^2(\R^3)}^{(10+4s)/5} + \|\nabla p\|_{L^2(\R^3)}^{(10+4s)/5} \right)dt 
 \leq C T^{\frac{13-3/s}{10}} \qquad T>1.
\end{align*}
By a convexity argument 
\begin{align*}
\int_0^T & H[u(t),p(t)]^{1+2s/5} dt\\
& = \int_0^T\left(\|u\|_{L^2(\R^3)}^{2} + \frac{1}{2}\|\nabla p\|_{L^2(\R^3)}^{2} \right)^{1+2s/5} dt\\
& \leq C\int_0^T\left(\|u\|_{L^2(\R^3)}^{(10+4s)/5} + \|\nabla p\|_{L^2(\R^3)}^{(10+4s)/5} \right)dt \\
& \leq C T^{\frac{13-3/s}{10}} \qquad T>1.
\end{align*}
On the other hand $t\mapsto H[u(t),p(t)]$ is non-increasing in time, so
$$
\int_0^T H[u(t),p(t)]^{1+2s/5} dt \geq \int_0^T H[u(T),p(T)]^{1+2s/5} dt = T H[u(T),p(T)]^{1+2s/5} .
$$
Putting the two previous inequalities together leads to
$$ T H[u(T),p(T)]^{1+2s/5}\leq C T^{\frac{13-3/s}{10}} \qquad T>1, $$
which yields the statement of the Theorem. This finished the proof.
\end{proof}

\section{Proof of Theorem \ref{thr.ws}}\label{sec:thm2}
\begin{proof}[Proof of Theorem ~\ref{thr.ws}]
Let us compute the time derivative of $H[(u(t),p(t))\vert (v(t),q(t))]$:
\begin{align}\label{ws.1}
 \frac{d}{dt}H[(u(t),p(t))\vert (v(t),q(t))] = 
 &- 2\IRd \nabla(u-v)\cdot (u\nabla p - v\nabla q) dx\\
 \nonumber
 &- \IRd |(-\Delta)^{s/2}\nabla(p-q)|^2 dx\\
 \nonumber
 &+ \IRd \nabla(p-q)\cdot \nabla (u^2 - v^2) dx
\end{align}
Let us consider the term
\begin{align*}
 & - 2\IRd \nabla(u-v)\cdot (u\nabla p - v\nabla q) dx\\
 &= -\IRd\left( 
 \na u^2\cdot\na p + \na v^2\cdot\na q - 2u\na v\cdot\na p - 2v\na u\cdot\na q 
 \right)dx\\
 &= -\IRd\left(
 \nabla (u^2 - v^2)\cdot\nabla(p-q) + \na v^2\cdot\na p + \na u^2\cdot\na q
 - 2u\na v\cdot\na p - 2v\na u\cdot\na q 
 \right)dx\\
 &= -\IRd \nabla (u^2 - v^2)\cdot\nabla(p-q) dx
  -2\IRd\left(
 (v-u)\na v\cdot\na p + (u-v)\na u\cdot\na q
 \right)dx\\
 &= -\IRd \nabla (u^2 - v^2)\cdot\nabla(p-q) dx\\
 &\quad -2\IRd (u-v)\left(
 \na (u-v)\cdot\na q - \na v\cdot\na (p-q)
 \right)dx\\
 &= -\IRd \nabla (u^2 - v^2)\cdot\nabla(p-q) dx 
 + \IRd (u-v)^2\Delta q dx + 2\IRd (u-v)\na (p-q)\cdot\na v dx
\end{align*}
Therefore \eqref{ws.1} becomes
\begin{align}\label{ws.2}
 \frac{d}{dt}H[(u(t),p(t)) &\vert (v(t),q(t))] + 
 \IRd |(-\Delta)^{s/2}\nabla(p-q)|^2 dx \\
 &= \IRd (u-v)^2\Delta q dx + 2\IRd (u-v)\cdot\na (p-q)\na v dx . \nonumber 
\end{align}
Let us bound the right-hand side of \eqref{ws.2}. It holds trivially
\begin{align}
 \IRd (u-v)^2\Delta q dx \leq 
 \left(\sup_{\R^3}\Delta q\right)
 H[(u,p) &\vert (v,q)] . \label{ws.3}
\end{align}
Let us now consider
\begin{align}\label{ws.4}
 & 2\IRd (u-v)\na (p-q)\cdot\na v\, dx \leq 
 \IRd (u-v)^2 dx + \IRd |\na v|^2 |\na(p-q)|^2 dx .
\end{align}
Let $2<\lambda<\frac{6}{3-2s}$, $\lambda'=\frac{\lambda}{\lambda-1}$.
H\"older inequality allows us to write
\begin{align*}
& \IRd |\na v|^2 |\na(p-q)|^2 dx \leq 
C \|\na v\|_{L^{2\lambda'}(\R^3)}^2
\|\na(p-q)\|_{L^{2\lambda}(\R^3)}^2 .
\end{align*}
By interpolation
\begin{align*}
& \IRd |\na v|^2 |\na(p-q)|^2 dx \leq 
C \|\na v\|_{L^{2\lambda'}(\R^3)}^2
\|\na(p-q)\|_{L^{2}(\R^3)}^{2\rho}
\|\na(p-q)\|_{L^{\frac{6}{3-2s}}(\R^3)}^{2(1-\rho)},
\end{align*}
for some $\rho \in (0,1)$. 
Moreover, thanks to the Sobolev embedding 
$H^s(\R^3)\hookrightarrow L^{\frac{6}{3-2s}}(\R^3)$ it follows
\begin{align*}
& \IRd |\na v|^2 |\na(p-q)|^2 dx \leq 
C \|\na v\|_{L^{2\lambda'}(\R^3)}^2
\|\na(p-q)\|_{L^{2}(\R^3)}^{2\rho}
\|(-\Delta)^{s/2}\na(p-q)\|_{L^{2}(\R^3)}^{2(1-\rho)}.
\end{align*}
Young's inequality yields
\begin{align*}
& \IRd |\na v|^2 |\na(p-q)|^2 dx \leq 
C \|\na v\|_{L^{2\lambda'}(\R^3)}^{2/\rho}
\|\na(p-q)\|_{L^{2}(\R^3)}^{2}
+\frac{1}{2}\|(-\Delta)^{s/2}\na(p-q)\|_{L^{2}(\R^3)}^{2}.
\end{align*}
From the above inequality and \eqref{ws.4} we deduce
\begin{align}\label{ws.5}
& 2\IRd (u-v)\na (p-q)\cdot\na v\, dx\\ 
\nonumber
&\leq \IRd (u-v)^2 dx + C \|\na v\|_{L^{2\lambda'}(\R^3)}^{2/\rho}
\IRd |\na(p-q)|^2 dx
+\frac{1}{2}\|(-\Delta)^{s/2}\na(p-q)\|_{L^{2}(\R^3)}^{2}\\
\nonumber
&\leq C(1 + \|\na v\|_{L^{2\lambda'}(\R^3)}^{2/\rho})
H[(u,p)\vert (v,q)] + \frac{1}{2}\|(-\Delta)^{s/2}\na(p-q)\|_{L^{2}(\R^3)}^{2}
\end{align}
Adding \eqref{ws.2}, \eqref{ws.3}, \eqref{ws.5} yields
\begin{align*}
 \frac{d}{dt}H[(u,p)\vert (v,q)] \leq C\left(1 +
 \sup_{\R^3}\Delta q
 + \|\na v\|_{L^{2\lambda'}(\R^3)}^{2/\rho}
 \right)H[(u,p)\vert (v,q)] .
\end{align*}
Since $2 < \lambda < \frac{6}{3-2s}$, then 
$\lambda' = \frac{\lambda}{\lambda-1} > \frac{6}{3+2s}$. Therefore it is
possible to choose $\lambda$ such that
$2\lambda' = \frac{12}{3+2s} + \nu$. As a consequence
$\nabla v\in L^\infty(0,\infty; L^{2\lambda'}(\R^3))$.
Gronwall's Lemma allows us to obtain the theorem's statement with
$$ K = C\left(1 +
 \sup_{\R^3\times(0,\infty)}\Delta q
 + \|\na v\|_{L^\infty(0,\infty; L^{\frac{12}{3+2s} + \nu}(\R^3))}^{2/\rho}
 \right) . $$
This finishes the proof.
\end{proof}

\section{Proof of Theorem \ref{cor.ex}}\label{sec:thm4}

\begin{proof}
The only point of the existence proof where the assumption $s\geq 3/4$ is used is the proof that $\urho$ is bounded in
$L^3(\R^3\times (0,T))$, which is in turn required to show that 
\begin{align}
\int_0^T\IRd ((\urho)^2 - v)(\prho - p)dx dt \to 0\quad\mbox{as }\rho\to 0, \label{lim.R3}
\end{align}
where $v$ is the weak limit of $(\urho)^2$ in $L^{3/2}(\R^3\times (0,T))$. However, in the case where $\R^3$ is replaced by $\T^3$,
it is possible to show the analogue of \eqref{lim.R3} without employing the bound for $\urho$ in $L^3(\R^3\times (0,T))$.
Indeed, under the assumption $s>1/2$, it holds that $H^{s+1}(\T^3)\hookrightarrow L^\infty(\T^3)$ compactly. This fact, combined with
the strong convergence of $\prho$ in $L^1(\T^3\times (0,T))$, yields the strong convergence of $\prho$ in 
$L^1(0,T; L^\infty(\T^3))$. In particular $\prho$ is Cauchy in $L^1(0,T; L^\infty(\T^3))$, i.e.~for every $\eps>0$ there is $\delta>0$
such that 
$$\|\prho - \psg\|_{L^1(0,T; L^\infty(\T^3))}<\eps \quad\mbox{for }\rho,\sigma<\delta. $$
Since $\urho$ is bounded in $L^\infty(0,T; L^2(\T^3))$ (thanks to the entropy inequality), it follows
\begin{align*}
 \int_0^T\int_{\T^3} ((\urho)^2 - (\usg)^2)(\prho - \psg)dx dt \leq 
 C\|\prho - \psg\|_{L^1(0,T; L^\infty(\T^3))} < C\eps\quad\rho,\sigma<\delta.
\end{align*}
This means that $(\urho)^2$ is Cauchy with respect to the norm $\|\cdot\|_{\mathcal A}$.
We point out that the positivity of the quadratic form $\mathcal A$ can also be showed through energy methods (by testing the second equation 
in \eqref{1L} against $p$ under assumption that $p(0)=0$), so it holds also in the torus case.

At this point, one proceeds like in the $\R^3$ case to show that from the property that $(\urho)^2$ is Cauchy with respect to $\|\cdot\|_{\mathcal A}$
it follows that $\urho$ is (up to subsequences) a.e.~convergent to $u$ in $\T^3\times (0,T)$, and therefore $v = u^2$. This finishes the proof.
\end{proof}



\begin{thebibliography}{11}

\bibitem{AJY19} F. Achleitner, A. Jüngel, and M. Yamamoto. Large-time asymptotics of a fractional drift-diffusion-Poisson system via the entropy method. Nonlin. Anal. 179 (2019), 270-293.

 \bibitem{AS08} L. Ambrosio and S. Serfaty. {\em A gradient flow approach to an evolution problem arising in superconductivity. } Comm. Pure Appl. Math. 61 (2008), no. 11, 1495-1539. 
 
  \bibitem{BIK15} P. Biler, C.Imbert, and G. Karch.{\em  The nonlocal porous medium equation: Barenblatt profiles and other weak solutions.} Arch. Ration. Mech. Anal. 215 (2015), no. 2, 497-529. 
  
  \bibitem{BKM10}   P. Biler, G. Karch and R. Monneau. {\em Nonlinear diffusion of dislocation density and self-similar solutions. }Comm. Math. Phys. 294 (2010), no. 1, 145-168. 

\bibitem{CGZ18} L.~Caffarelli, M.~Gualdani, N.~Zamponi. {\em A non-local porous media equation. }
 Submitted for publication 2018, arxiv:1805.08666v2.
 
\bibitem{CSV13}
L. Caffarelli, F. Soria and J.L. Vazquez. {\em Regularity of solutions of the fractional porous medium flow.} J. Eur. Math. Soc. 15 (2013), no. 5, 1701-1746. 

 \bibitem{CV11_II}
 L. Caffarelli and J.L. Vazquez. {\em Asymptotic behaviour of a porous medium equation with fractional diffusion. }Discrete Contin. Dyn. Syst. 29 (2011), no. 4, 1393-1404. 


\bibitem{CV15}
 L. Caffarelli and J.L. Vazquez. {\em  Regularity of solutions of the fractional porous medium flow with exponent 1/2. }Algebra i Analiz 27 (2015), no. 3, 125 - 156; translation in St. Petersburg Math. J. 27 (2016), no. 3, 437-460

\bibitem{CV11}
 L. Caffarelli, J.L. Vazquez. {\em Nonlinear porous medium flow with fractional potential pressure.} Arch. Ration. Mech. Anal. 202 (2011), no. 2, 537-565. 

\bibitem{D10} C. Dafermos. {\em Hyperbolic conservation laws in continuum physics.} Grundlehren der Mathematischen Wissenschaften 325. Springer-Verlag, Berlin, 2010.

\bibitem{GL97} G. Giacomin, J. Lebowitz.{\em Phase segregation dynamics in particle systems with long range interactions. I. Macroscopic limits.}  J. Statist. Phys. 87 (1997), no. 1-2, 37-61. 

\bibitem{GL98} G. Giacomin, J. Lebowitz.{\em  Phase segregation dynamics in particle systems with long range interactions. II. Interface motion. }SIAM J. Appl. Math. 58 (1998), no. 6, 1707-1729. 

 \bibitem{GLM00} G. Giacomin, J. Lebowitz and R. Marra. {\em Macroscopic evolution of particle systems with short- and long-range interactions.} Nonlinearity 13 (2000), no. 6, 2143-2162.
 
 
\bibitem{Jue2015}
A.~J\"ungel. {\em The boundedness-by-entropy method for cross-diffusion systems.} Nonlinearity 28 (6) (2015), 1963-2001.

\bibitem{Jue2016}
A.~J\"ungel. {\em Entropy methods for diffusive partial differential equations.} {Springer}, 2016.

\bibitem{LMG01} 
P.L. Lions, S. Mas-Gallic
{\em Une methode particulaire deterministe pour des equations diffusives non lineaires. }
C. R. Acad. Sci. Paris Sér. I Math. 332 (2001), no. 4, 369-376. 


\bibitem{L34} J. Leray. {\em Sur le mouvement d'un liquide visqueux emplissant l’espace.} Acta Math. 63(1):193-248, 1934.

\bibitem{P59}  G. Prodi. {\em Un teorema di unicita' per le equazioni di Navier-Stokes.} Ann. Mat. Pura Appl. 48:173-182, 1959.

\bibitem{R91} F. Rezakhanlou. {\em Hydrodynamic limit for attractive particle systems on Zd.} Comm. Math. Phys. 140 (1991), no. 3, 417-448. 

\bibitem{S62} J. Serrin. {\em The initial value problem for the Navier-Stokes equations. }In Nonlinear Problems (Proc. Sympos., Madison, Wis., 1962), 69-98.

\bibitem{SV14}
S. Serfaty and J.L. Vazquez. {\em A mean field equation as limit of nonlinear diffusions with fractional Laplacian operators. } Calc. Var. Partial Differential Equations 49 (2014), no. 3-4, 1091-1120. 



 

\end{thebibliography}
\end{document}